\documentclass[10pt]{amsart}

\usepackage{amsmath, amsthm, amscd, amsfonts, amssymb, graphicx, color}
\usepackage{latexsym}
\usepackage{amscd}
\usepackage{amsthm}
\usepackage{amssymb}
\usepackage{enumerate}
\usepackage{mathabx}

\theoremstyle{plain}
\newtheorem{theorem}{Theorem}
\newtheorem{corollary}[theorem]{Corollary}

\theoremstyle{remark}

\newcommand{\C}{\mathbb{C}}
\newcommand{\R}{\mathbb{R}}
\newcommand{\N}{\mathbb{N}}
\newcommand{\Z}{\mathbb{Z}}
\newcommand*{\Hom}{\ensuremath{\mathrm{Hom\,}}}
\newcommand*{\rank}{\ensuremath{\mathrm{rank\,}}}

\begin{document}

\title{Characterization of Classes of Polynomial Functions}
\author{J.~M.~Almira, L.~Sz\'ekelyhidi}



\let\thefootnote\relax\footnotetext{The research was supported by the Hungarian National Foundation for Scientific Research (OTKA),   Grant No. NK-81402.}

\address{Departamento de Matem\'{a}ticas, Universidad de Ja\'{e}n, E.P.S. Linares,  C/Alfonso X el Sabio, 28, 23700 Linares, Spain}
\email{jmalmira@ujaen.es}
\address{Institute of Mathematics, University of Debrecen, Egyetem t\'er 1, 4032 Debrecen, Hungary --- Department of Mathematics, University of Botswana, 4775 Notwane Rd. Gaborone, Botswana }
\email{lszekelyhidi@gmail.com}

\maketitle
\begin{abstract}
In this paper some classes of local polynomial functions on abelian groups are characterized by the properties of their variety. For this characterization we introduce a numerical quantity depending on the variety of the local polynomial only. Moreover, we show that the known characterization of polynomials among generalized polynomials can be simplified: a generalized polynomial is a polynomial if and only if its variety contains finitely many linearly independent additive functions.
\end{abstract}

\section{Introduction}
Polynomials on commutative groups play a basic role in functional equations and in spectral synthesis. The most common definition of polynomial functions depends on Fr\'echet's Functional Equation (see \cite{Fre09, MazOrl34, MR0001560}). Given a commutative group $G$ we denote by $\C G$ the {\it group algebra} of $G$, which is the algebra of all finitely supported complex valued functions defined on $G$. Besides the linear operations (addition and multiplication by scalars) the multiplication is defined by convolution
\begin{equation*}
\mu*\nu(x)=\sum_{y\in G} \mu(x-y)\nu(y)
\end{equation*}
for each $x$ in $G$. With these operations $\C G$ is a commutative complex algebra with identity $\delta_o$, where $o$ is the zero element in $G$ and for each $y$ in $G$ we use the notation $\delta_y$ for the characteristic function of the singleton $\{y\}$. Elements of this algebra of the form 
\begin{equation*}
\Delta_y=\delta_{-y}-\delta_0
\end{equation*}
with $y$ in $G$ are called {\it differences}.
\vskip.3cm

Using the notation $\mathcal C(G)$ for the linear space of all complex valued functions on $G$, it is a module over $\C G$ with the obvious definition
\begin{equation*}
\mu*f(x)=\sum_{y\in G} f(x-y)\mu(y)
\end{equation*}
for each $x$ in $G$. 
\vskip.3cm

The function $f:G\to\C$ is called a {\it generalized polynomial} of degree at most $n$, if $n$ is a natural number and
\begin{equation}\label{Frech1}
\Delta_{y_1,y_2,\dots,y_{n+1}}*f=0\,,
\end{equation}
where we use the notation $\Delta_{y_1,y_2,\dots,y_{n+1}}$ for the convolution product 
\begin{equation*}
\Delta_{y_1}*\Delta_{y_2}*\dots*\Delta_{y_{n+1}}\,.
\end{equation*}
The smallest $n$ with this property is called the {\it degree} of $f$. In \cite{MR0265798} Djokovi\v c proved that condition \eqref{Frech1}, which is called {\it Fr\'echet's Functional Equation}, is equivalent to the condition
\begin{equation}\label{Frech2}
\Delta_y^{n+1}*f=0\,,
\end{equation}
where $\Delta_y^{n+1}=\Delta_{y_1,y_2,\dots,y_{n+1}}$ with $y=y_1=y_2=\dots=y_{n+1}$. We note that sometimes \eqref{Frech2} is also called Fr\'echet's Functional Equation. In \cite{MR0265798} Djokovi\v c showed that the two functional equations \eqref{Frech1} and \eqref{Frech2} are equivalent  for complex valued functions on every abelian group (see also \cite{Sze14b}).
\vskip.3cm

Polynomials of degree at most one, which vanish at zero, are called {\it additive functions}. They are characterized by the equation
\begin{equation*}
a(x+y)=a(x)+a(y)\,,
\end{equation*}
that is, they are exactly the homomorphisms of $G$ into the additive group of complex numbers. All additive functions on $G$ form a linear space, which is denoted by $\Hom(G,\C)$.
\vskip.3cm

There is a vast literature on different types of polynomials, which play a basic role in the theory of functional equations. In \cite{MR2165414} M.~Laczkovich studies the relations of diverse concepts of polynomials. The reader will find further references and results in this respect in \cite{MazOrl34, MR2582364,  MR954205, MR2433311,  MR1113488, MR0001560}.
\vskip.3cm

A special class of generalized polynomials is formed by those functions, which belong to the function algebra generated by the additive functions and the constants. These functions are simply called {\it polynomials}. Hence the general form of a polynomial is
\begin{equation}\label{poly}
p(x)=P\bigl(a_1(x),a_2(x),\dots,a_n(x)\bigr)
\end{equation}
for each $x$ in $G$, where $a_1,a_2,\dots,a_n:G\to\C$ are additive functions and $P:\C^n\to\C$ is an ordinary polynomial in $n$ variables. In the case $G=\R^n$ or $G=\C^n$ it is well-known (see e.g. \cite{MR1113488}), that every continuous generalized polynomial is a polynomial, in fact, it is an ordinary polynomial. In particular, in this case the additive functions in \eqref{poly} are continuous, assuming that they are linearly independent, which we always may suppose.
\vskip.3cm

The following theorems hold true (see e.g. \cite[Theorem 2. and Theorem 3.]{MR2167990}, \cite[Theorem 8.]{MR2433311}, \cite[Theorem 4.]{MR2968200}).

\begin{theorem}\label{char1}
Let $G$ be an abelian group. A generalized polynomial on $G$ is a polynomial if and only if the dimension of $\tau(f)$ is finite.
\end{theorem}

\begin{theorem}\label{genpol}
Let $G$ be an abelian group. Every generalized polynomial on $G$ is a polynomial if and only if the dimension of $\Hom(G,\C)$ is finite.
\end{theorem}

If $G$ is finitely generated, then it is easy to see that every generalized polynomial on $G$ is a polynomial (see e.g. \cite[Theorem 2. and Theorem 3.]{MR2167990}). 
\vskip.3cm

In \cite{Lacz13} M.~Laczkovich introduced the concept of local polynomials. A function $f:G\to\C$ is called a {\it local polynomial}, if its restriction to every finitely generated subgroup is a polynomial. By the previous remark, every generalized polynomial is a local polynomial, however, as it is shown in \cite{Lacz13}, there are local polynomials, which are not generalized polynomials. 

\section{Results}

For the sake of simplicity a generalized polynomial, which is not a polynomial will be called a {\it fake polynomial}.

\begin{theorem}\label{findim}
Let $G$ be an abelian group. If $f:G\to\C$ is a fake polynomial,  then the linear space spanned by all additive functions in $\tau(f)$ is infinite dimensional.
\end{theorem}

\begin{proof}
As $f$ is fake polynomial, hence its degree $n$ is at least two. We have the unique representation
\begin{equation}\label{fake}
f(x)=A_n(x,x,\dots,x)+A_{n-1}(x,x,\dots,x)+\dots+A_2(x,x)+A_1(x) +C\,,
\end{equation}
where $A_k:G^k\to\C$ is $k$-additive and symmetric for $k=1,2,\dots,n$, and $C$ is a constant. Suppose first that there is an integer $k$ with $1\leq k\leq n$ such that the functions $x\mapsto A_j(x,x,\dots,x)$ are polynomials for $j=k, k+1,\dots,n$. Then $k\geq 3$. Let 
\begin{equation*}
g(x)=A_n(x,x,\dots,x)+A_{n-1}(x,x,\dots,x)+\dots+A_k(x,x,\dots,x)
\end{equation*}
for each $x$ in $G$. Then $g$ is a polynomial, hence $\tau(g)$ is finite dimensional, consequently $\tau(g)$ contains finitely many linearly independent additive functions. As obviously we have 
\begin{equation*}
\tau(f-g)\subseteq \tau(f)+\tau(g)\,,
\end{equation*}
hence in order to prove or statement it is enough to show that in $\tau(f-g)$ there are infinitely many linearly independent additive functions. This argument shows that we may suppose in \eqref{fake} that $x\mapsto A_n(x,x,\dots,x)$ is a fake polynomial.
\vskip.3cm

An easy computation (taking $n-1$-th differences of $f$) shows that all additive functions of the form
\begin{equation*}
x\mapsto A_n(x,y_2,\dots,y_n)
\end{equation*}
with $y_2,y_3,\dots,y_n$ in $G$ are included in $\tau(f)$. Suppose that the linear space $A_f$ of all additive functions in $\tau(f)$ is finite dimensional. Then we show by induction on $k$ that all functions 
\begin{equation*}
x\mapsto A_n(x,x,\dots,x,y_{k+1},y_{k+2},\dots,y_n)
\end{equation*}
are polynomials, where $1\leq k\leq n$. This statement is obvious for $k=1$, as the function $A_n(x,y_2,\dots,y_n)$ is additive, hence it is a polynomial for each $y_2,\dots,y_n$ in $G$.
\vskip.3cm

 Suppose now that $k\geq 2$ and we have proved that all functions 
 \begin{equation*}
 x\mapsto A_n(x,x,\dots,x,y_{k},y_{k+1},\dots,y_n)
 \end{equation*}
 are polynomials, and we prove this for $k+1$ instead of $k$. 
 \vskip.3cm
 
 Let $\{a_i(x)\}_{i=1}^N$ be a basis of $A_f$. Then we have that there exists points $\{x_j\}_{j=1}^N$ in $G$ such that the matrix $\bigl(a_i(x_j)\bigr)$ is regular,  and we have
 \[
 A_n(x,y_2,\cdots,y_n)=\sum_{i=1}^N c_i(y_2,\cdots,y_n)a_i(x)
 \]
with certain functions $c_i:G^{n-1}\to\C$, $i=1,2,\dots,N$. It follows
 \[
 A_n(x_j,y_2,\cdots,y_n)=\sum_{i=1}^N c_i(y_2,\cdots,y_n) a_i(x_j)
 \]
 and, using that $(a_i(x_j))$ is regular, we obtain, by Cramer's Rule
 \[
 c_i(y_2,\dots,y_n)=\frac{D_i(y_2,\dots,y_n)}{D}\,.
 \]
Here $D$ is the determinant of the matrix $(a_i(x_j))$, and $D_i$ is the determinant obtained from $D$ by replacing its $i$-th column with the vector whose $j$-th component is $A_n(x_j,y_2,\cdots,y_n)$.  It follows that $D_i(y_2,y_3,\dots,y_n)$ is a linear combination of the $A_n(x_j,y_2,\cdots,y_n)$'s for $j=1,2,\dots,n$. In particular, by our assumption, and by the symmetry of $A_n$, for each $2\leq k<n$ when substituting $y_2=y_3=\dots=y_k=x$, the function
\begin{equation*}
x\mapsto D_i(x,x,\dots,x,y_{k+1},\dots,y_n)
\end{equation*}
is a polynomial.
\vskip.3cm

On the other hand, we have
 \[
 A_n(x,y_2,\cdots,y_n)=\frac{1}{D}\sum_{i=1}^N D_i(y_2,\dots,y_n) a_i(x)\,.
 \]
 Thus, the substitution $y_2=\cdots=y_k=x$ and the symmetry of $A_n$ implies that 
 \begin{equation*}
 A(x,x,x,\cdots,x,y_{k+1},...,y_n) = \frac{1}{D}\sum_{i=1}^N D_i(x,x,\dots,x,y_{k+1},\dots,y_n) a_i(x)\,,
 \end{equation*}
 which is a polynomial. Our proof is complete. 
\end{proof}

\begin{corollary}
Let $G$ be an abelian group. Then a generalized polynomial on $G$ is a polynomial if and only if the linear space generated by all additive functions in its variety is finite dimensional.
\end{corollary}

In connection with this Corollary the following question arises: is it true that for each fake polynomial $f$ of degree at least three the variety generated by the differences $\Delta_y*f$ for each $y$ in $G$ contains a fake polynomial? In the affirmative case this would imply the statement of our Corollary above. However, this statement is not true as it is shown by the following example.
\vskip.3cm

Let $G=\Z_{\omega}$ denote the direct sum of $\omega$ copies of $\Z$. In other words, $\Z_{\omega}$ is the set of all integer valued finitely supported functions on $\N$ with the pointwise addition. Obviously, $\Z_{\omega}$ is the so called {\it weak direct product} of $\omega$ copies of $\Z$, hence it is a subgroup of the (complete) direct product of $\omega$ copies of $\Z$. We denote the elements of $G$ by $x=(x_i)_{i\in\N}$.
\vskip.3cm

We define the function $f:G\to\C$ by
\begin{equation*}
f(x)=\sum_{i\in \N} x_i^3
\end{equation*}
whenever $x$ is in $G$. Let $y$ be in $G$, then we can proceed as follows:
\begin{equation*}
f(x+y)-f(x)=\sum_{i\in \N} (3x_i^2 y_i+3 x_i y_i^2+y_i^3)=3\sum_{i\in \N}x_i^2 y_i+\sum_{i\in \N} (3 x_i y_i^2+y_i^3)\,.
\end{equation*} 
Obviously, both sums are finite. The second sum is additive plus constant, hence it is a polynomial. In the first sum $x_i^2$ has a nonzero coefficient if and only if $y_i\ne 0$, hence for only a finite set of $i$'s, which is independent of $x$, it is depending on $y$, only. It follows that $f$ is a polynomial. On the other hand, taking second differences of $f$, it is clear that $\tau(f)$ includes all the projections $x\mapsto x_i$, which are linearly independent. It follows, that $\tau(f)$ is infinite dimensional, hence $f$ is not a polynomial, by Theorem \ref{char1}.
\vskip.3cm

Let $G$ be an abelian group and let $f:G\to\C$ be a local polynomial. Let $\tau(f)$ denote the variety generated by $f$. For each positive integer $r$ we introduce the following quantity:
\begin{equation*}
d_f(r)=\sup_{\rank H\leq r} \dim \tau(f|_H)\,,
\end{equation*}
where the supremum is taken over all finitely generated subgroups $H$ of rank at \hbox{most $r$,} and $f|_H$ denotes the restriction of $f$ onto $H$. Clearly, $\dim \tau(f|_H)$ is finite for each finitely generated subgroup $H$. It is also obvious that the function $d_f$ is increasing.
\vskip.3cm

We have the following result.

\begin{theorem}
Let $G$ be an abelian group and let $f:G\to\C$ be a local polynomial. Then we have the following statements:
\begin{enumerate}[i)]
\item $d_f(r)=+\infty$ for each $r\geq 2$ if and only if $f$ is not a generalized polynomial;
\item $d_f(r)<+\infty$ for each $r$ and $\lim_{r\to \infty} d_f(r)=+\infty$ if and only if $f$ is a fake polynomial;
\item $\lim_{r\to \infty} d_f(r)<+\infty$ if and only if $f$ is a polynomial.
\end{enumerate}
\end{theorem}

\begin{proof}
Suppose that $f$ is not a generalized polynomial. It follows that for each positive integer $n$ there exist $x,y$ in $G$ such that 
\begin{equation*}
\Delta_y^{n+1}*f(x)\ne 0\,.
\end{equation*} 
Let $H$ denote the subgroup generated by $x$ and $y$, then the restriction $f|_H$ of $f$ to $H$ is a polynomial of degree at least $n+1$. It follows that there are elements $y_1,y_2,\dots,y_n$ in $H$ such that the function $\Delta_{y_1,y_2,\dots,y_k}*f|_H$ is a polynomial of degree $n-k$ for $k=1,2,\dots,n$, hence the functions
\begin{equation*}
f|_H, \Delta_{y_1}*f|_H,\dots,\Delta_{y_1,y_2,\dots,y_n}*f|_H
\end{equation*}
are linearly independent and they belong to $\tau(f|_H)$. Hence the dimension of $\tau(f|_H)$ is at least $n+1$, which gives $d_f(r)=+\infty$ for each $r\geq 2$.
\vskip.3cm

Suppose that $f$ is a fake polynomial. Then, by Theorem \ref{findim}, there is a sequence $(a_k)_{k=1}^{\infty}$ of linearly independent additive functions in $\tau(f)$. Let $N$ be a positive integer and we take elements $x_1,x_2,\dots,x_N$ in $G$ such that the matrix $\bigl(a_k(x_j)\bigr)_{k,j=1}^N$ is regular. Let $L$ be any finitely generated subgroup including the elements $x_1,x_2,\dots,x_N$. Then the restriction $f|_L$ of $f$ to $L$ is a polynomial, hence $\tau(f|_L)$ is finite dimensional -- it is equal to the linear span of all translates of $f|_L$. The restrictions of the $a_k$'s to $L$ are linearly independent for $k=1,2,\dots,N$, and they are linear combinations of some translates of $f$ -- let $K$ be the set of all elements of $G$, which occur as translating elements to obtain the $a_k$'s for $k=1,2,\dots,N$. Let $H$ denote the subgroup generated by the $L$ and $K$, then $H$ is a finitely generated subgroup of $G$ and the restrictions of the $a_k$'s belong to $\tau(f|_H)$ for $k=1,2,\dots,N$. It follows that the dimension of $\tau(f|_H)$ is at least $N$, hence $\lim_{r\to\infty} d_r(f)=+\infty$. On the other hand, obviously $d_r(f)$ is finite for each $r$.
\vskip.3cm

Finally, suppose that $f$ is a polynomial. Then it has the form
\begin{equation}\label{polform}
f(x)=P\bigl(a_1(x),a_2(x),\dots,a_k(x)\bigr)
\end{equation}
with some ordinary polynomial $P:\C^k\to\C$ and linearly independent additive functions $a_1,a_2,\dots,a_k:G\to\C$. We have, by Taylor Formula,
\begin{equation}\label{Tay}
f(x+y)=\sum_{|\alpha|\leq \deg P} \frac{1}{\alpha !}\, \partial^{\alpha} P\bigl(a_1(x),\dots,a_k(x)\bigr) a_1(y)^{\alpha_1}\dots a_k(y)^{\alpha_k}
\end{equation}
for each $x,y$ in $G$, where $\alpha=(\alpha_1,\alpha_2,\dots,\alpha_k)$ is a multi-index with 
$$
|\alpha|=\alpha_1+\alpha_2+\dots+\alpha_k\,.
$$ 

It is known (see e.g. \cite[Theorem 3.2.7]{MR1113488},  p.~33.) that the functions
\begin{equation*}
y\mapsto a_1(y)^{\alpha_1}\dots a_k(y)^{\alpha_k}
\end{equation*}
are linearly independent for different choices of the  multi-index $\alpha$ (here we use the convention $0^0=1$). It follows that the functions
\begin{equation*}
x\mapsto \partial^{\alpha} P\bigl(a_1(x),\dots,a_k(x)\bigr)
\end{equation*}
are linear combinations of translates of $f$, hence they belong to $\tau(f)$. Moreover, by the above equation, these functions generate $\tau(f)$. The number of the different functions of these type is not greater than $(\deg f+1)^k$. It follows that $\lim_{r\to \infty} d_f(r)\leq  (\deg f+1)^k<+\infty$.
\vskip.3cm

The theorem is proved.
\end{proof}


\end{document}